\documentclass[11pt]{amsart}
\usepackage{amsmath,amssymb,amsthm,bm,url}

\newtheorem{thm}{Theorem}[section]
\newtheorem{prop}[thm]{Proposition}
\newtheorem{lem}[thm]{Lemma}
\newtheorem{cor}[thm]{Corollary}
\theoremstyle{definition}
\newtheorem{dfn}[thm]{Definition}

\newtheorem{rem}[thm]{Remark}
\numberwithin{equation}{section}

\begin{document}
\title{Polars and antipodal sets of generalized $s$-manifolds}

\author{Shinji Ohno}
\address{Department of Mathematics College of Humanities and Sciences,
Nihon University,
3-25-40 Sakurajosui, Setagaya-ku, Tokyo 156-8550, Japan.}
\email{ohno.shinji@nihon-u.ac.jp}

\author{Takashi Sakai}
\address{Department of Mathematical Sciences,
Tokyo Metropolitan University,
1-1 Minami-Osawa, Hachioji, Tokyo 192-0397, Japan}
\email{sakai-t@tmu.ac.jp}

\date{\today}



\thanks{The second author was supported by JSPS KAKENHI Grant Number JP21K03250.
}

\maketitle
\begin{abstract}
{
In this paper, we introduce the notion of generalized $s$-manifolds,
which is a generalization of symmetric spaces.
Then we give a method to construct generalized $s$-manifolds and present some typical examples.
We study polars and antipodal sets of generalized $s$-manifolds
aiming to extend the results on compact symmetric spaces due to Chen and Nagano.
}

\end{abstract}


\section{Introduction}
\label{sec:introduction}

The geometry of Riemannian symmetric spaces has been studied from the 1920s
initiated by E. Cartan.
Riemannian symmetric spaces are distinguished Riemannian manifolds and are intensively studied
not only in differential geometry but also other related fields, such as topology, harmonic analysis,
and representation theory.
Moreover, there are many works on generalizations of symmetric spaces.
Loos \cite{Loos} and Nagano \cite{Nagano} focused on the algebraic structure
of geodesic symmetries of Riemannian symmetric spaces, and studied symmetric spaces
without Riemannian metrics (Definition~\ref{dfn:symmetric spaces}). 
This concept was generalized to the study of $k$-symmetric spaces and (regular) $s$-manifolds
(cf. \cite{Ledger}, \cite{LO}, \cite{Kowalski}).
Furthermore, in 1981, Lutz \cite{Lutz} introduced the notion of $\Gamma$-symmetric spaces,
which has a symmetry isomorphic to a finite abelian group $\Gamma$ at each point (Definition~\ref{dfn:Gamma-symmetric space}).
In the case where $\Gamma$ is isomorphic to $\mathbb{Z}_2 \times \mathbb{Z}_2$, there are several works
in the literature (e.g. \cite{Bahturin-Goze}, \cite{Goze-Remm}, \cite{Kollross}).
In these works, $\Gamma$ is supposed to be finite abelian,
because those are related to graded Lie algebras.
In this paper, we introduce the notion of generalized $s$-manifolds
whose symmetry group $\Gamma$ is not necessarily finite and abelian,
moreover $\Gamma$ can be a Lie group.
Although we have already introduced generalized $s$-manifolds in \cite{OST},
we give the precise definition in this paper (Definition~\ref{dfn:generalized s-manifold}).
Then we present several typical examples of generalized $s$-manifolds.
It is well-known that a Riemannian symmetric space can be obtained from a Riemannian symmetric pair.
We give a method to construct homogeneous generalized $s$-manifolds
from $\Gamma$-symmetric triples (Section~\ref{sec:Gamma-symmetric_triple}).
The symmetry of a symmetric space has the algebraic structure as a quandle.
In some studies, quandles are considered as discrete symmetric spaces, and the theory of symmetric spaces is used as a basis for the study of quandles (cf. \cite{KNOT}, \cite{Tamaru}, etc.).
We show that a generalized $s$-manifold has a structure of a $G$-family of quandles
(Proposition~\ref{prop:G-family of quandles}).

Chen and Nagano studied polars, meridians, antipodal sets and $2$-numbers,  
which are determined from the structure of compact symmetric spaces in a series of papers since 1978 (\cite{CN}, \cite{Chen-Nagano1988}, etc.). 
The results of these studies are called the Chen--Nagano theory. 
It is notable that a connected compact irreducible symmetric space is determined by a pair of a polar and a meridian.
The Euler number of a compact symmetric space is evaluated from above by its $2$-number. 
Our study is motivated to extend the results on compact symmetric spaces due to Chen and Nagano.
In Section~\ref{sec:polars_and_antipodal_sets_of_generalized_s-manifolds}, we define $\gamma$-polars, polars, and antipodal sets of generalized $s$-manifolds.
Then we obtain the invariant named \textit{antipodal number}
as the supremum of the cardinalities of antipodal sets of the generalized $s$-manifold.
We show inequalities of the antipodal number from the above 
by the sum of the antipodal numbers of its polars and $\gamma$-polars (Theorem~\ref{thm:inequality of antipodal number}).
The maximal antipodal sets of $\Gamma$-symmetric spaces, where $\Gamma$ is finite abelian,
were studied in \cite{Terauchi}, \cite{Okubo}, \cite{Quast-Sakai} and \cite{Amann}.

\section{Definition of generalized $s$-manifolds}
\label{sec:definition_of_generalized_s-manifolds}

In this section, first  we recall the definition of Riemannian symmetric spaces
and some generalized notions.
As a generalization of those notions,
we introduce generalized $s$-manifolds.

A connected Riemannian manifold $(M,g)$ is called a \textit{Riemannian symmetric space} 
if for each point $x\in M$ there exists an isometry $s_x$ of $M$ which 
fixes $x$ and reverses all geodesics passing through $x$. 
Here, $s_x$ is called the \textit{geodesic symmetry} at the point $x\in M$.
Loos \cite{Loos} and Nagano \cite{Nagano} 
studied the following symmetric spaces without  Riemannian metrics.
Hereafter, we set
$$
F(\phi, M)=\{ y \in M \mid \phi(y) = y\}, \quad 
F(\Phi, M)=\{ y \in M \mid \phi(y) = y \ (\forall \phi \in \Phi)\}
$$
for a diffeomorphism $\phi$ of a manifold $M$
and a subgroup $\Phi$ of the diffeomorphism group $\mathrm{Diff}(M)$ of $M$. 

\begin{dfn} \label{dfn:symmetric spaces}
Let $M$ be a $C^\infty$-manifold and $\mu : M \times M \to M$ a $C^\infty$-map.
For each $x \in M$, we define the map $s_x : M \to M$ by $s_x(y) := \mu(x,y) \ (y \in M)$.
Then $M$ is called a \textit{symmetric space},
if it satisfies the following conditions.
\begin{enumerate}
\item[(S1)] $s_{x}(x)=x\quad (\forall x\in M)$. 
\item[(S2)] $s_{x}\circ s_{x}=\mathrm{id}_{M} \quad (\forall x \in M)$. 
\item[(S3)] $s_{x}\circ s_{y}=s_{s_{x}(y)}\circ s_{x} \quad (\forall x,y \in M)$.
\item[(S4)] For each $x \in M$, $x$ is an isolated fixed point of $s_{x}$,  
that is, $x$ is an isolated point in $F(s_x, M)$.\end{enumerate}
\end{dfn}
For a Riemannian symmetric space $(M,g)$, 
the geodesic symmetry $s_x$ at each point $x \in M$ is a diffeomorphism since it is an isometry and clearly satisfies conditions (S1)--(S4). 
Therefore $(M, g)$ is a symmetric space in the sense of Definition~\ref{dfn:symmetric spaces}.
Nagano \cite{Nagano} studied the category of symmetric spaces and 
defined the notion of morphism between two symmetric spaces $M$ and $N$.
A $C^\infty$-map $f : M \to N$ is called a \textit{morphism} if $f$ and
the set $\{s_x\}$ of point symmetries of $M$ are compatible, 
i.e.,  $f \circ s_x = s_{f(x)} \circ f$ holds for any $x \in M$.
The condition (S3) means that, for each $x \in M$,
the point symmetry $s_x$ is an automorphism of $M$.
The condition (S3) is also the condition (Q3) of the definition of quandles described below. 

\begin{dfn}[\cite{Joyce}] \label{dfn:Quandle}
Let $X$ be a set equipped with a binary operation $* : X \times X \to X$.
Then $(X, *)$ is called a \textit{quandle},
if it satisfies the following conditions.
\begin{enumerate}
\item[(Q1)] $x*x=x \quad (\forall x \in X)$.
\item[(Q2)] For each $x,y \in X$, there exists a unique $z \in X$ such that $z*y=x$.
\item[(Q3)] $(x*y)*z = (x*z)*(y*z) \quad (\forall x,y,z \in X)$.
\end{enumerate}
\end{dfn}

For a symmetric space $M$ in the sense of Definition~\ref{dfn:symmetric spaces}, 
we define a binary operation $* : M \times M \to M$ on $M$ by $x*y := s_y(x)$.
Then $(M, *)$ is a quandle.

Ledger \cite{Ledger} studied Riemannian $s$-manifolds generalizing Riemannian symmetric spaces. 
In \cite{Kowalski77}, a regular $s$-manifold is defined without a Riemannian metric as follows. 

\begin{dfn}[\cite{Kowalski77} Definition~1] \label{dfn:regular s-manifold}
Let $M$ be a $C^\infty$-manifold and $\mu : M \times M \to M$ a $C^\infty$-map.
For each $x \in M$, we define the map $s_x : M \to M$ by $s_x(y) := \mu(x,y) \ (y \in M)$.
Then $M$ is called a \textit{regular $s$-manifold},
if it satisfies the following conditions.
\begin{enumerate}
\item $s_x(x) = x \quad (\forall x \in M)$.
\item For each $x \in M$, $s_x : M \to M$ is a diffeomorphism. 
\item $s_x \circ s_y = s_{s_x(y)} \circ s_x \quad (\forall x,y \in M)$.
\item For each $x \in M$, $x$ is an isolated fixed point of $s_{x}$,  
that is, $x$ is an isolated point in $F(s_x, M)$.
\end{enumerate}
\end{dfn}

In Kowalski's lecture note \cite{Kowalski}, 
a regular $s$-manifold is defined by the  following condition (4')
instead of the condition (4) in Definition~\ref{dfn:regular s-manifold}.
\begin{enumerate}
\item[(4')] For each $x\in M$, the differential map $(ds_{x})_{x}: T_{x} M \to T_{x} M$ has no fixed vectors except the zero vector.
\end{enumerate}
The conditions (4) and (4') are equivalent if there exists a connection which is invariant under the point symmetries $\{s_{x}\}_{x\in M}$.
However, the existence of such a connection is non-trivial.
An $s$-manifold satisfying the condition (4') is called a \textit{tangentially regular} $s$-manifold in \cite{Kowalski77}.
A tangentially regular $s$-manifold has a connection which is invariant under the point symmetries.
Thus, a tangentially regular $s$-manifold is a regular $s$-manifold.
A regular $s$-manifold gives a structure of quandle
by defining the binary operation $*$
similarly to the case of a symmetric space.

Whereas a regular $s$-manifold has the symmetry of a single transformation $s_{x}$ at each point $x$,
a $\Gamma$-symmetric space, which was introduced by Lutz (\cite{Lutz}),
has the symmetry of a group $\Gamma$ at each point.

\begin{dfn}[\cite{Lutz}] \label{dfn:Gamma-symmetric space}
Let $M$ be a connected $C^\infty$-manifold
and $\Gamma$ a finite abelian group.
A \textit{$\Gamma$-symmetric structure} on $M$
is a family $\mu = \{\mu^\gamma \}_{\gamma \in \Gamma}$
of smooth maps $\mu^\gamma : M \times M \to M$ which satisfies the following conditions.
\begin{enumerate}
\item For each $x \in M$,
the map 
$$\varphi_x:\Gamma \to \mathrm{Diff}(M);\ \gamma \mapsto \gamma_x := \mu^\gamma(x,\cdot)$$
is an injective homomorphism.
\item Each $x \in M$ is isolated in the fixed point set
$F(\varphi_{x}(\Gamma), M)$.
\item For each $x \in M$ and $\gamma \in \Gamma$,
$\gamma_x$ is an automorphism of $\mu$, i.e.
$$
\mu^\delta \big( \gamma_x(y), \gamma_x(z) \big)
= \gamma_x \big(\mu^\delta(y,z)\big) \quad
\text{for all} \ y,z \in M\ \text{and for all}\ \delta \in \Gamma.
$$
\end{enumerate}
\end{dfn}

In the preceding research on $\Gamma$-symmetric spaces, many studies have focused on the relationship with graded Lie algebras (\cite{Bahturin-Goze}).
On the other hand, our aim is to investigate
polars, antipodal sets and antipodal numbers defined from symmetric transformations. 
Therefore, it is not necessary to assume finiteness or commutativity of $\Gamma$.
In addition, there are many examples which have symmetries of various groups, as we will see later. 
Therefore, we define the following notion of generalized $s$-manifolds.  
\begin{dfn} \label{dfn:generalized s-manifold}
Let $M$ be a $C^{\infty}$-manifold and $\Gamma$ a group.
For each point $x \in M$,
let $\varphi_{x}:\Gamma \to \mathrm{Diff}(M)$ be a group homomorphism which satisfies the following conditions.

\begin{enumerate}
\item[(1)] For each $\gamma \in \Gamma$, 
the map $\mu^\gamma: M \times M \to M;\ (x,y) \mapsto \varphi_x(\gamma)(y)$ is $C^\infty$. 
In particular, if $\Gamma$ is a Lie group, 
then the map $\mu : \Gamma \times M \times M \to M;\ (\gamma,x,y) \mapsto \varphi_x(\gamma)(y)$ is $C^\infty$.

\item[(2)] Each $x \in M$ is an isolated fixed point of the action of $\varphi_{x}(\Gamma)$ on $M$,
that is, $x$ is an isolated point of 
$F(\varphi_x(\Gamma), M)$.

\item[(3)] For any $x, y \in M$ and $\gamma, \delta \in \Gamma$,
 $$
 \varphi_{x}(\gamma) \circ \varphi_{y}(\delta) \circ \varphi_{x}(\gamma)^{-1} 
=\varphi_{\varphi_{x}(\gamma)(y)}(\gamma \delta \gamma^{-1}).
$$
\end{enumerate}
Then
$(\Gamma, \{\varphi_x \}_{x \in M})$ is called a \textit{generalized $s$-structure} on $M$, 
and we call $(M, \Gamma, \{\varphi_x \}_{x \in M})$ a \textit{generalized $s$-manifold}.
\end{dfn}

We call $\varphi_x(\gamma) \in \mathrm{Diff}(M)$ the \textit{symmetric transformation} at $x \in M$
for $\gamma \in \Gamma$.
The symmetric transformation $\varphi_x(\gamma)$ is 
also denoted by
$\gamma_{x}$. 
For each point $x \in M$, the subgroup 
$\Gamma_x := \varphi_x(\Gamma)=\{\varphi_x(\gamma) \mid \gamma \in \Gamma\}
= \{ \gamma_{x} \mid \gamma \in \Gamma \}$
of $\mathrm{Diff}(M)$ is
called the \textit{symmetric transformation group} at $x \in M$.
With this notation, the condition (3) is expressed as
\begin{align*}
\gamma_{x} \circ \delta_{y} \circ  \gamma_{x}^{-1} 
=(\gamma \delta \gamma^{-1})_{\gamma_x (y)}
\end{align*}
for any $x, y\in M$ and $\gamma, \delta \in \Gamma$.
In particular, if $\Gamma$ is an abelian group, 
then we have
\begin{align*}
\gamma_{x} \circ \delta_{y} \circ \gamma_{x}^{-1} 
= \delta_{\gamma_x (y)}
\end{align*}
and the condition is equivalent to the condition (3) of Definition~\ref{dfn:Gamma-symmetric space}.  
Furthermore, when $\Gamma$ is generated by a single element, i.e. $\Gamma =\langle s \rangle$, 
this condition is equivalent to
\begin{align*}
s_{x} \circ s_{y} \circ s_{x}^{-1} 
=s_{s_x (y)}.
\end{align*}
In particular, a symmetric space defined by Definition~\ref{dfn:symmetric spaces} has a generalized $s$-structure where $\Gamma=\mathbb{Z}_{2}$.
Thus, the generalized $s$-manifold is a generalization of the symmetric space. 
More generally, it can be observed that $k$-symmetric spaces have the generalized $s$-structure where $\Gamma=\mathbb{Z}_{k}$ (cf. \cite{Kowalski}). 

\begin{rem}
In the case of a $\Gamma$-symmetric space $M$,
the symmetric transformation group $\Gamma_x$ is isomorphic to $\Gamma$
for every point $x \in M$ by the condition (1) in Definition~\ref{dfn:Gamma-symmetric space}.
On the other hand, in Definition~\ref{dfn:generalized s-manifold}
of a generalized $s$-manifold $(M, \Gamma, \{\varphi_x \}_{x \in M})$,
injectivity of $\varphi_x$ is not assumed.
Hence $\Gamma_{x}$ may not be isomorphic to $\Gamma$  for a point $x \in M$.
Moreover $\Gamma_{x}$ and $\Gamma_{y}$ may not be isomorphic for $x$ and $y$  in $M$.
In order to consider a \textit{subspace},
defined in Definition~\ref{dfn:subspace},
as a generalized $s$-manifold, 
injectivity of $\varphi_x$ cannot be assumed. 
\end{rem}

\begin{dfn}
A generalized $s$-structure $(\Gamma, \{\varphi_x \}_{x \in M})$ on a manifold $M$ is said to be \textit{effective} if $\varphi_x$ is injective for each point $x$ of $M$. 
\end{dfn}

Similarly to the fact that a symmetric space can be considered as a quandle, 
we can see that a generalized $s$-manifold is a $G$-family of quandles. 
A $G$-family of quandles is a generalized notion of quandle introduced in \cite{IIJO} to study handlebody-knots.

\begin{dfn}[\cite{IIJO}] \label{dfn:Gamma-family of quandles}
Let $X$ be a set and $G$ a group. 
Let $*^{g} : X \times X \to X\ (g\in G)$ be a family of binary operations 
on $X$ which satisfies the following conditions.
\begin{enumerate}
\item[(GQ1)] For any $x \in X$ and any $g\in G$, $x*^{g}x=x$.
\item[(GQ2)] For any $x,y \in X$, any $g, h\in G$ and the identity element $e\in G$, 
$$
x*^{gh}y=(x*^{g}y)*^{h}y \quad \text{and} \quad x*^{e}y=x.
$$
\item[(GQ3)] For any $x,y,z \in X$ and any $g, h\in G$, 
$$
(x*^{g}y)*^{h}z = (x*^{h}z)*^{h^{-1}gh}(y*^{h}z).
$$
\end{enumerate}
Then $(X, \{*^{g}\}_{ g\in G})$ is called a \textit{$G$-family of quandles}. 
\end{dfn}

We see that  a generalized $s$-structure $(\Gamma, \{\varphi_x \}_{x \in M})$ of $M$ induces a  structure of a $\Gamma$-family of quandles. 
\begin{prop}\label{prop:G-family of quandles}
Let $(M, \Gamma, \{\varphi_x \}_{x \in M})$ be a generalized $s$-manifold.
For each $\gamma \in \Gamma$, 
we define the binary operation $*^{\gamma}: M\times M \to M$ on $M$ by $x*^{\gamma}y=\varphi_{ y}(\gamma^{-1})(x)$.
Then $(M, \{*^{\gamma}\}_{ \gamma \in \Gamma})$ is a $\Gamma$-family of quandles.
\end{prop}
\begin{proof} 
(GQ1) For any $x \in M$ and any $\gamma\in \Gamma$, $x*^{\gamma}x=\varphi_{x}(\gamma^{-1})(x)=x$.

(GQ2) For any $x,y \in M$ and any $\gamma, \delta\in \Gamma$, 
\begin{align*}
x*^{\gamma \delta}y
&=\varphi_{y}((\gamma \delta)^{-1})(x) 
=\varphi_{y}(\delta^{-1} \gamma^{-1})(x) \\
&=\varphi_{y}(\delta^{-1})\circ\varphi_{y}( \gamma^{-1})(x)
=(x*^{\gamma}y)*^{\delta}y
\end{align*}
and 
$$
x*^{e}y=\varphi_{y}(e)(x)=x
$$
hold.

(GQ3) For any $x,y,z \in M$ and any $\gamma,\delta \in \Gamma $,
we have 
\begin{align*}
(x*^{\gamma}y)*^{\delta}z &= \varphi_{z} (\delta^{-1})\circ  \varphi_{y} (\gamma^{-1}) (x)\\
&= \varphi_{\varphi_{z} (\delta^{-1})(y)} (\delta^{-1} \gamma^{-1} \delta ) \circ \varphi_{z} (\delta^{-1} )(x)\\
&= \varphi_{y*^{\delta} z} (\delta^{-1} \gamma^{-1} \delta ) (x*^{\delta}z)\\
&= (x*^{\delta}z)*^{\delta^{-1}\gamma \delta}(y*^{\delta}z).
\end{align*}
\end{proof}
\begin{rem}
The definition of $\Gamma$-family of quandles is given by the right action,
whereas that of generalized $s$-manifold is given by the left action. 
Hence, we use the inverse of $\gamma$ to define the binary operation $*^{\gamma}$ from $\gamma$. 
\end{rem}


\section{Constructions and examples}
\label{sec:constructions_and_examples}

In this section, we provide a method to construct generalized $s$-manifolds,
and give some concrete examples.

\subsubsection*{Example $1$} 

Let $T^2 = \mathbb{R}^2/\mathbb{Z}^2$ be the $2$-dimensional torus.
The canonical projection from $\mathbb{R}^2$ to $T^2$ is expressed as
$\mathbb{R}^2 \to T^2; \bm{x} \mapsto [\bm{x}]$,
and the unit element of $T^2$ is denoted by $o=[(0,0)]$.

Let $GL(2,\mathbb{Z})$ denote the group of $2 \times 2$ unimodular matrices:
$$
GL(2,\mathbb{Z}) = \left\{ \gamma =\left(\begin{array}{cc} a & b \\ c & d \end{array}\right)
\Big|  \begin{array}{l} a,b,c,d \in \mathbb{Z} \\ ad-bc=\pm 1 \end{array} \right\}.
$$
Since $GL(2,\mathbb{Z})$ acts on $\mathbb{R}^2$ preserving the integer lattice $\mathbb{Z}^2$,
it induces automorphisms of $T^2$.
If fact, $GL(2,\mathbb{Z})$ is the automorphism group of $T^2$ as a Lie group.

Let $\Gamma$ be a subgroup of $GL(2,\mathbb{Z})$.
For each $[\bm{x}] \in T^2$ and $\gamma \in \Gamma$, we define the diffeomorphism $\gamma_{[\bm{x}]}$
of $T^2$ by
$$
\gamma_{[\bm{x}]}([\bm{y}]) = [\gamma(\bm{y}-\bm{x})+\bm{x}] \qquad ([\bm{y}] \in T^2). 
$$
Then $\varphi_{[\bm{x}]} : \Gamma \to \mathrm{Diff}(T^2); \gamma \mapsto \gamma_{[\bm{x}]}$
is a homomorphism.
Note that $\gamma_{[\bm{x}]} = \tau_{[\bm{x}]} \circ \gamma_{o} \circ \tau_{-[\bm{x}]}$,
where $\tau_{[\bm{x}]}$ is the translation of $T^2$ by $[\bm{x}] \in T^2$.
We can verify that $\{ \varphi_{[\bm{x}]}\}_{[\bm{x}]\in T^2}$ satisfies the conditions (1) and (3)
of Definition~\ref{dfn:generalized s-manifold}.
In addition, if $o$ is isolated in $F(\varphi_o(\Gamma), T^2)$,
then $[\bm{x}]$ is also isolated in $F(\varphi_{[\bm{x}]}(\Gamma), T^2)$ for every $[\bm{x}] \in T^2$.
Therefore $\{ \varphi_{[\bm{x}]}\}_{[\bm{x}]\in T^2}$ satisfies  the condition (2)
of Definition~\ref{dfn:generalized s-manifold},
and $(\Gamma, \{ \varphi_{[\bm{x}]}\}_{[\bm{x}]\in T^2})$ is a generalized
$s$-structure on $T^2$.

For example, the following subgroups $\Gamma$ of $GL(2,\mathbb{Z})$ define
generalized $s$-structures on $T^2$.
\begin{enumerate}
\item $\Gamma=\left\langle \left[
    \begin{array}{cc}
         -1&0  \\
         0&-1
    \end{array}\right] \right\rangle \cong \mathbb{Z}_{2}$
\item $\Gamma=\left\langle 
    \left[
    \begin{array}{cc}
         -1&0  \\
         0&1
    \end{array}\right], 
    \left[
    \begin{array}{cc}
         1&0  \\
         0&-1
    \end{array}\right]
    \right\rangle \cong \mathbb{Z}_{2}\times\mathbb{Z}_{2}$
\item $\Gamma=\left\langle 
    \left[
    \begin{array}{cc}
         0&1  \\
         1&0
    \end{array}\right], 
    \left[
    \begin{array}{cc}
         0&-1  \\
         -1&0
    \end{array}\right]
    \right\rangle \cong \mathbb{Z}_{2}\times\mathbb{Z}_{2} $
\item $\Gamma=\left\langle 
    \left[
    \begin{array}{cc}
         -1&0  \\
         0&1
    \end{array}\right], 
    \left[
    \begin{array}{cc}
         0&1  \\
         1&0
    \end{array}\right]
    \right\rangle \cong \mathrm{D}_{4} $
\item $\Gamma=\left\langle 
    \left[
    \begin{array}{cc}
         -1&-1  \\
         1&0
    \end{array}\right]
    \right\rangle \cong \mathbb{Z}_{3} $
\item $\Gamma=\left\langle 
    \left[
    \begin{array}{cc}
         0&-1  \\
         1&0
    \end{array}\right]
    \right\rangle \cong \mathbb{Z}_{4} $
\item $\Gamma=\left\langle 
    \left[
    \begin{array}{cc}
         0&-1  \\
         1&1
    \end{array}\right]
    \right\rangle \cong \mathbb{Z}_{6} $
\item $\Gamma=\left\langle 
    \left[
    \begin{array}{cc}
         2&1  \\
         1&1
    \end{array}\right]
    \right\rangle \cong \mathbb{Z}$
\end{enumerate}

In the case (1), $\Gamma$ induces the structure of a symmetric space on $T^2$ as a compact Lie group.

In the case  (2), the symmetric transformation group $\Gamma_o$ at $o \in T^2$ is generated
by the reflections over the two axes.  This structure of a $\mathbb{Z}_{2}\times\mathbb{Z}_{2}$-symmetric space
is given as the direct product of two symmetric spaces. 

In the case (3), the symmetric transformation group $\Gamma_o$ at $o$ is generated
by the reflections over the two lines $x+y=0$ and $x-y=0$.
This structure of a $\mathbb{Z}_{2}\times \mathbb{Z}_{2}$-symmetric space
is different from that of (2).
Indeed, in Section~\ref{sec:polars_and_antipodal_sets_of_generalized_s-manifolds},
we will observe that these two generalized $s$-structures have different maximal antipodal sets.

In the case (4), $\Gamma$ is the dihedral group $\mathrm{D}_{4}$,
which is generated by that of (2) and (3).
Note that $\mathrm{D}_{4}$ is non-abelian.

In the cases  (5), (6) and (7), the $2$-torus $T^2$ has structures of $k$-symmetric spaces for $k=3,4,6$ (cf. \cite{BGPC}).

In the case (8), this action of $\mathbb{Z}$ on $T^2$ is known as 
a chaotic map called \textit{Arnold's cat map}.

\smallskip
This example shows that a manifold may admit various generalized $s$-structures.


\subsection{$\Gamma$-symmetric triple}
\label{sec:Gamma-symmetric_triple}

Lutz \cite{Lutz} introduced $\Gamma$-symmetric triples
to construct $\Gamma$-symmetric spaces (see also \cite{Bahturin-Goze}, \cite{Goze-Remm}, \cite{Quast-Sakai_2022}, and \cite{Terauchi}).
We generalize it to construct generalized $s$-manifolds.

\begin{dfn}
Let $G$ be a connected Lie group
and $\Gamma$ a finite subgroup or a compact Lie subgroup of the automorphism group $\mathrm{Aut}(G)$ of $G$.
Let $K$ be a closed subgroup of $G$ satisfying
\begin{equation}\label{eq:Gamma-triple}
F(\Gamma,G)_0 \subset K \subset F(\Gamma,G),
\end{equation}
 where $F(\Gamma,G)_0$ denotes the connected component of $F(\Gamma,G)$ containing the identity element $e \in G$.
Then the triple $(G,K,\Gamma)$ is called a \textit{$\Gamma$-symmetric triple}.
\end{dfn}

From a $\Gamma$-symmetric triple $(G,K,\Gamma)$,
we provide a $G$-equivariant generalized $s$-structure on the homogeneous space $G/K$ as follows.
Note that, since $G$ is connected, $\mathrm{Aut}(G)$ is a Lie group in a natural manner.
And,
$F(\Gamma, G)$ is a closed Lie subgroup of $G$.

\begin{thm}\label{thm:Gamma-symmetric_triple}
Let $(G,K,\Gamma)$ be a $\Gamma$-symmetric triple.
For each point $x=g_x K \in G/K$ and for each $\gamma \in \Gamma$, 
we define the diffeomorphism $\varphi_{x}(\gamma)$ of $G/K$ by
\begin{equation} \label{eq:Gamma-symmetric_triple}
\varphi_{x}(\gamma)(gK) := g_x\gamma(g_x^{-1}g)K \qquad (gK \in G/K).
\end{equation}
Then $\varphi_x : \Gamma \to \mathrm{Diff}(G/K)$ is an injective homomorphism,
and $(\Gamma, \{\varphi_x\}_{x\in G/K})$ is a $G$-equivariant, effective generalized $s$-structure on $G/K$.
\end{thm}

Although we can prove the theorem in a similar way with Proposition~2.4 in \cite{Quast-Sakai_2022},
we give a proof below for completeness.

\begin{proof}
It is clear that,  for each $x\in G/K$ and $\gamma \in \Gamma$,
the map $\varphi_x(\gamma)$ is a (well-defined) diffeomorphism of $G/K$,
and we can verify that $\varphi_x : \Gamma \to \mathrm{Diff}(G/K)$ is a group homomorphism.
Note that, at the origin $o:=eK$ of $G/K$,
the diffeomorphism $\varphi_o(\gamma)$ is given by
\begin{equation}\label{eq:symmetric_transformation_at_the_origin}
\varphi_o(\gamma)(gK) = \gamma(g)K\quad (\forall gK \in G/K).
\end{equation}
We then have the following $G$-equivariance law: 
\begin{equation}\label{eq:G-equivariance_law}
\varphi_x(\gamma) = L_{g_x} \circ \varphi_{o}(\gamma) \circ L_{g_x^{-1}}\quad
(\forall x=g_x K \in G/K),
\end{equation}
where $L_g$ denotes the left action of $g\in G$ on $G/K$.
Thus, to show the injectivity of $\varphi_x$ for all $x \in G/K$,
it is sufficient to verify that $\varphi_o$ is injective.
For any $\gamma \in \Gamma$,
the differential $(d\gamma)_e$ at $e\in G$ is an automorphism of the Lie algebra $\mathfrak{g}$ of $G$,
that is $\Gamma$ also acts on $\mathfrak{g}$.
The condition (\ref{eq:Gamma-triple}) implies that the Lie algebra $\mathfrak{k}$ of $K$
coincides with that of $F(\Gamma, G)$, i.e.,
\begin{equation}\label{eq: Lie_algebra_of_K}
\mathfrak{k}=\{X\in  \mathfrak{g} \mid (d\gamma)_e(X)=X \ \text{for all}\ \gamma\in\Gamma\}.
\end{equation}
Since $\Gamma$ is compact, we can take a $\Gamma$-invariant inner product on $\mathfrak{g}$.
Then we have the orthogonal direct sum decomposition $\mathfrak{g} = \mathfrak{k} \oplus \mathfrak{m}$,
where $\mathfrak{m}$ is the orthogonal complement of $\mathfrak{k}$ in $\mathfrak{g}$.
Note that $\mathfrak{m}$ is invariant under the action of $\Gamma$ on $\mathfrak{g}$.
The differential $(d\pi)_e : \mathfrak{g} \to T_o(G/K)$ of the canonical projection $\pi:G\to G/K$
at $e \in G$ is surjective, and its kernel is $\mathfrak{k}$.
Hence the restriction of $(d\pi)_e$ to $\mathfrak{m}$,
denoted by $(d\pi)_e|_{\mathfrak{m}}:\mathfrak{m}\to T_o(G/K)$,
is a linear isomorphism.
The definition (\ref{eq:symmetric_transformation_at_the_origin}) of $\varphi_o(\gamma)$ implies
$\varphi_o(\gamma) \circ \pi = \pi \circ \gamma$, hence
\begin{equation}\label{eq:differential_of_symmetric_transformation}
d(\varphi_o(\gamma))_o \circ (d\pi)_e|_{\mathfrak{m}}
= (d\pi)_e|_{\mathfrak{m}} \circ (d\gamma)_e|_{\mathfrak{m}} 
\end{equation}
for all $\gamma \in \Gamma$,
where $d(\varphi_o(\gamma))_o$ is the differential of $\varphi_o(\gamma)$ at $o \in G/K$.
Therefore, if $\varphi_o(\gamma)$ is the identity on $G/K$,
then
$(d\gamma)_e|_{\mathfrak{m}}$ is the identity on $\mathfrak{m}$. 
Since $(d\gamma)_e$ is also the identity on $\mathfrak{k}$ by (\ref{eq: Lie_algebra_of_K}),
then $(d\gamma)_e$ is the identity on $\mathfrak{g}$. 
Therefore $\gamma$ is the identity of $G$ since $G$ is connected.
This shows that $\varphi_o$ has the trivial kernel, hence $\varphi_o$ is injective.

We shall verify that $(\Gamma, \{\varphi_x\}_{x\in G/K})$ satisfies
the three conditions of Definition~\ref{dfn:generalized s-manifold}.

The condition (1) is clear.
For each $\gamma\in\Gamma$, the multiplication map 
$$
\mu^\gamma:G/K\times G/K\to G/K;\
(x,y)\mapsto \varphi_x(\gamma)(y) = g_x\gamma(g_x^{-1}g_y)K
$$
is $C^\infty$.
In particular, when $\Gamma$ is a compact Lie subgroup of $\mathrm{Aut}(G)$, 
the map
$$
\mu : \Gamma \times G/K \times G/K \to G/K;\
(\gamma,x,y) \mapsto \varphi_x(\gamma)(y) = g_x\gamma(g_x^{-1}g_y)K
$$
is $C^\infty$.

Next we show the condition (2).
In the above argument, $\mathfrak{m}$ is the orthogonal complement of $\mathfrak{k}$ in $\mathfrak{g}$.
Hence, by (\ref{eq: Lie_algebra_of_K}) and (\ref{eq:differential_of_symmetric_transformation}),
for any $X\in T_o(G/K)$ with $X\neq 0$,
there exists $\gamma\in\Gamma$ such that $d(\varphi_o(\gamma))_o(X) \neq X$. 
Since $\Gamma$ is compact, we can take a Riemannian metric on $G/K$
which is invariant under the action of the symmetric transformation group $\varphi_o(\Gamma)$ at $o$.
We take an open ball $B$ in $T_o(G/K)$ around $o$ with sufficiently small radius
so that the Riemannian exponential map $\mathrm{Exp}_o$ of $G/K$ at $o$ is a diffeomorphism
from $B$ onto $\mathrm{Exp}_o(B)$.
For any $\gamma\in\Gamma$, we have
$$
\varphi_o(\gamma) \circ \mathrm{Exp}_o = \mathrm{Exp}_o \circ d(\varphi_o(\gamma))_o,
$$
since $\varphi_o(\gamma)$ is an isometry of $G/K$ fixing $o$.
Thus we have
$$
\mathrm{Exp}_o(B) \cap F(\varphi_o(\Gamma), G/K) = \{o\},
$$
hence $o$ is isolated in $F(\varphi_o(\Gamma), G/K)$.
This implies the condition (2) by the transitive action of $G$ on $G/K$
and the $G$-equivariance (\ref{eq:G-equivariance_law}) of $\{\varphi_x\}_{x\in G/K}$.

Finally we can verify the condition (3) by a direct calculation as follows.
For $x=g_xK, y=g_yK, z=g_zK \in G/K$,
\begin{align*}
&\varphi_{\varphi_x(\gamma)(y)}(\gamma\delta\gamma^{-1}) \big( \varphi_x(\gamma)(z) \big)\\
&= \big( \varphi_{g_x\gamma(g_x^{-1}g_y)K}(\gamma\delta\gamma^{-1})\big)
\big(g_x\gamma(g_x^{-1}g_z)K\big) \\
&= \big(g_x\gamma(g_x^{-1}g_y)\big)  (\gamma\delta\gamma^{-1})
\Big( \big(g_x\gamma(g_x^{-1}g_y)\big)^{-1} \big(g_x\gamma(g_x^{-1}g_z)\big) \Big) K \\
&= \big(g_x\gamma(g_x^{-1}g_y)\big)  (\gamma\delta\gamma^{-1})
\big( \gamma(g_y^{-1}g_x ) \gamma(g_x^{-1}g_z) \big) K \\
&= g_x\gamma(g_x^{-1}g_y) (\gamma\delta)(g_y^{-1}g_z) K \\
&= g_x\gamma \big( g_x^{-1}g_y\, \delta(g_y^{-1}g_z) \big) K \\
&= \varphi_x(\gamma) \big( g_y\, \delta(g_y^{-1}g_z) K \big) \\
&= \varphi_x(\gamma) \big( \varphi_y(\delta)(z) \big).
\end{align*}
Thus we complete the proof of the proposition.
\end{proof}

Similarly we can prove the following proposition.

\begin{prop}\label{prop:gamma-family of quandle}
Let $G$ be a group and $\Gamma$ a subgroup of $\mathrm{Aut}(G)$ of $G$.
Let $K$ be a subgroup of $G$ satisfying
$$
K \subset F(\Gamma,G).
$$
For each $\gamma \in \Gamma$, we define
$*^\gamma : G/K \times G/K \to G/K$ by
$$
x *^\gamma y =  g_y \gamma^{-1}(g_y^{-1} g_x)K \quad \text{for } x=g_xK \text{ and } y=g_yK \text{ in } G/K.
$$
Then $(G/K, \{*^\gamma \}_{\gamma\in\Gamma})$ is a $G$-equivariant $\Gamma$-family of quandle.
\end{prop}


\subsubsection*{Example $2$}
We introduce an example of a generalized $s$-manifold where $\Gamma$ is a Lie group.
Let $G$ be a compact connected semisimple Lie group and $T$ a maximal torus of $G$. 
The quotient space $G/T$ is called a flag manifold. 
We define a generalized $s$-structure on $G/T$ with $\Gamma=T$. 
The natural projection $\pi: G \to G/T$ is given by $\pi(g)=[g]:=gT$. 
Let $\Gamma'$ be the subgroup of $\mathrm{Aut}(G)$ defined by 
$
\Gamma'=\{ \mathrm{I}_{a} \mid a\in T\}. 
$
Here $\mathrm{I}_{a}$ is the inner automorphism of $G$ with respect to $a\in T$.
Then we immediately see that $T\subset F(\Gamma', G)$.
We also have
\begin{align*}
\mathrm{Lie}(F(\Gamma', G))
&=\{ X\in \mathfrak{g} \mid \mathrm{Ad}(a)X=X \ (\forall a\in T)\}\\
&=\{ X\in \mathfrak{g} \mid [Y, X]=0 \ (\forall Y\in \mathrm{Lie}(T))\}
=\mathrm{Lie}(T).
\end{align*}
Thus we have $F(\Gamma', G)_{0} \subset T$.
Then, $(G, T, \Gamma')$ is a $\Gamma'$-symmetric triple.
From Theorem~\ref{thm:Gamma-symmetric_triple},
$G/T$ has the generalized $s$-structure $(\Gamma', \{\varphi_{x}\}_{x\in G/T})$.
Then the natural homomorphism $f:T \to \Gamma'; a\mapsto \mathrm{I}_{a}$ is surjective. 
Hence, $(G/T, T, \{\varphi_{x}\circ f\}_{x\in G/T})$ is a generalized $s$-manifold.

\subsubsection*{Example $3$}
More generally, 
we show that 
a generalized flag manifold also has a generalized $s$-structure with a certain torus group $\Gamma$.
Let $G$ be a compact connected semisimple Lie group.
The Lie algebra of $G$ is denoted by $\mathfrak{g}$. 
For $X\in \mathfrak{g}\setminus \{0\}$, 
the adjoint orbit $M:=\mathrm{Ad}(G)X$ is called a generalized flag manifold. 
Then $M$ is diffeomorphic to $G/G_{X}$, 
where $G_{X}:=\{g\in G \mid \mathrm{Ad}(g)X=X\}$.  
We define a generalized $s$-structure on $M \cong G/G_{X}$ with $\Gamma=\mathrm{Z}(G_{X})_{0}$.
Here $\mathrm{Z}(G_{X})_0$ is the identity component of
the center $\mathrm{Z}(G_{X})$ of the isotropy subgroup $G_{X}$.
Let $\Gamma'$ be the subgroup of $\mathrm{Aut}(G)$
defined by 
$
\Gamma'=\{ \mathrm{I}_{a} \mid a\in \mathrm{Z}(G_{X})_{0}\}. 
$
Here $\mathrm{I}_{a}$ is the inner automorphism of $G$ with respect to $a\in \mathrm{Z}(G_{X})_{0}$.
Then we see that $G_{X}\subset F(\Gamma', G)$.
We also have
\begin{align*}
\mathrm{Lie}(G_{X})&=\{Y \in \mathfrak{g}\mid \mathrm{ad}(Y)X=0\},  \\
\mathrm{Lie}(F(\Gamma', G))
&=\{ Y\in \mathfrak{g} \mid \mathrm{Ad}(a)Y=Y \ (\forall a \in \mathrm{Z}(G_{X})_{0})\}.
\end{align*}
For each $Y\in \mathrm{Lie}(F(\Gamma', G))$, 
since $\mathrm{Ad}(\exp (t X))Y=Y \ (\forall t\in \mathbb{R})$,
we have $[X, Y]=0$.
Thus we have $Y\in \mathrm{Lie}(G_{X})$ and $\mathrm{Lie}(F(\Gamma', G)) \subset \mathrm{Lie}(G_{X})$.
Therefore, $(G, G_X,\Gamma')$ is a $\Gamma'$-symmetric triple.
From Theorem~\ref{thm:Gamma-symmetric_triple},
$G/G_{X}$ has the generalized $s$-structure $(\Gamma', \{\varphi_{x}\}_{x\in G/G_{X}})$.
Then the natural homomorphism $f:\mathrm{Z}(G_{X})_{0} \to \Gamma';  a \mapsto \mathrm{I}_{a}$ is surjective. 
Hence, $(G/G_{X}, Z(G_{X})_{0}, \{\varphi_{x}\circ f\}_{x\in G/G_{X}})$ is a generalized $s$-manifold.

This structure is related to the study of the intersection of real forms of complex flag manifolds (cf.~\cite{IIOST}).  

\subsubsection*{Example $4$}
We can also construct generalized $s$-manifolds using compact symmetric triads. 
We take a compact connected semisimple Lie group $G$ and two involutive automorphisms $\theta_{1}, \theta_{2}$ of $G$. 
Let $K_{1}, K_{2}$ be closed subgroups of $G$.  
If 
$$
F(\theta_{i}, G)_{0} \subset K_{i}\subset F(\theta_{i}, G)\quad (i=1, 2)
$$
holds, 
then $(G, K_{1}, K_{2})$ is called  a compact symmetric triad.
That is, 
$(G, K_{1})$ and $(G, K_{2})$ are compact symmetric pairs. 
Thus, $G/K_{i}$ has the structure of symmetric space.
Indeed, 
we can define the point symmetry $s_{gK_{i}}$ on $G/K_{i}$ by 
$$
s_{gK_{i}}(g'K_{i}):=g\theta_{i}(g^{-1}g')K_{i}\quad (g' K_{i}\in G/K_{i})
$$
for each point  $gK_{i} \in G/K_{i}$. 

For a given compact symmetric triad $(G, K_{1}, K_{2})$, 
we construct the following generalized $s$-structure on the quotient manifold $G/H$ where $H=K_{1}\cap K_{2}$ . 
Let 
$\Gamma = \langle \theta_{1}, \theta_{2} \rangle$ be the subgroup of $\mathrm{Aut}(G)$ generated by $\theta_{1}$ and $\theta_{2}$.
Note that $\Gamma$ is not always finite.
We define the symmetric transformation $\varphi_{gH}(\gamma)$ on $G/H$ by 
$$
\varphi_{gH}(\gamma)(g'H)=g\gamma(g^{-1}g')H\quad \quad (g' H\in G/H)
$$
for each $gH\in G/H$ and $\gamma \in \Gamma$
as same as (\ref{eq:Gamma-symmetric_triple}).
Then we prove that 
$(\Gamma, \{\varphi_{gH}\}_{gH \in G/H})$ is a generalized $s$-structure on $G/H$.
When $\Gamma$ is finite, we can apply Theorem~\ref{thm:Gamma-symmetric_triple}.
Even if $\Gamma$ is infinite, $(G/H, \{\ast ^{\gamma}\}_{\gamma \in \Gamma})$ is a $\Gamma$-family of quandles from Proposition~\ref{prop:gamma-family of quandle}.
In particular, (GQ3) of Definition~\ref{dfn:Gamma-family of quandles} holds for $\{\ast ^{\gamma}\}_{\gamma \in \Gamma}$.
This implies that $(\Gamma, \{\varphi_{gH}\}_{gH \in G/H})$ satisfies (3) of Definition~\ref{dfn:generalized s-manifold}.
Then, (1) of Definition~\ref{dfn:generalized s-manifold} is clear, furthermore
we can verify (2) of Definition~\ref{dfn:generalized s-manifold} directly.
Indeed,
we show that the action of $\varphi_{gH}(\Gamma)$ has $gH$ as an isolated fixed point for each $gH\in G/H$.  
First, we show that $o:=eH$ is an isolated fixed point of the action of $\varphi_{o}(\Gamma)$. 
We set 
$$
\mathfrak{k}_{i}=\{X\in \mathfrak{g} \mid (d\theta_{i})_{e}(X)=X\}, \quad 
\mathfrak{m}_{i}=\{X\in \mathfrak{g} \mid (d\theta_{i})_{e} (X)=-X\}\  (i=1,2).
$$
Then, the Lie algebra $\mathfrak{h}$ of $H$ coincides with $\mathfrak{k}_{1}\cap \mathfrak{k}_{2}$, 
thus we have 
$\mathrm{Ad}(H) (\mathfrak{m}_{1}+\mathfrak{m}_{2}) =\mathfrak{m}_{1}+\mathfrak{m}_{2}$.
When $\theta_{1}$ and $\theta_{2}$ are commutative, 
then
$$
\mathfrak{g}=\mathfrak{h} \oplus (\mathfrak{m}_{1}\cap \mathfrak{m}_{2})\oplus (\mathfrak{k}_{1}\cap \mathfrak{m}_{2})\oplus (\mathfrak{m}_{1}\cap \mathfrak{k}_{2})
=\mathfrak{h} \oplus (\mathfrak{m}_{1} + \mathfrak{m}_{2})
$$
holds.
Even if $\theta_{1}$ and $\theta_{2}$ are not commutative,
we can prove that $\mathfrak{m}_{1}+\mathfrak{m}_{2}$ coincides with the orthogonal complement of $\mathfrak{h}$ in $\mathfrak{g}$ with respect to the negative of the Killing form of $\mathfrak{g}$.
Therefore, we obtain the orthogonal direct sum decomposition
$\mathfrak{g}= \mathfrak{h}\oplus (\mathfrak{m}_{1}+\mathfrak{m}_{2})$,
where the subspace $\mathfrak{m}_{1}+\mathfrak{m}_{2}$ is $\mathrm{Ad}(H)$-invariant.
This means that $G/H$ has a normal homogeneous metric,
and the tangent space $T_{o}(G/H)$ can be identified with $\mathfrak{m}_{1}+\mathfrak{m}_{2}$.
Then the Riemannian exponential map $\mathrm{Exp}_o : \mathfrak{m}_{1}+\mathfrak{m}_{2} \to G/H$ at $o$
is given as
$$
\mathrm{Exp}_o(X) = \exp (X)H \in G/H \quad (X \in \mathfrak{m}_{1}+\mathfrak{m}_{2}).
$$
We take an $\mathrm{Ad}(H)$-invariant neighborhood $U$ around $0 \in \mathfrak{m}_{1}+\mathfrak{m}_{2}$
so that $\mathrm{Exp}_o|_U$ is an $H$-equivariant diffeomorphism from $U$ to $\mathrm{Exp}_o(U)$.
Since $(d\theta_{1})_{e}$ and $(d\theta_{2})_{e}$ are orthogonal transformations, we can assume $(d\theta_{1})_{e}(U)=U$ and $(d\theta_{2})_{e}(U)=U$. 
In fact, an open ball centered at $0\in \mathfrak{m}_{1}+\mathfrak{m}_{2}$ that is contained in $U$ has that property. 
We suppose that there exists a vector $X\in U$ such that
$\varphi_{o}(\gamma)(\mathrm{Exp}_o(X)) = \mathrm{Exp}_o(X)$ holds for all $\gamma \in \Gamma$.
Since each $\gamma\in \Gamma$ is an automorphism on $G$, we have
\begin{align*}
\mathrm{Exp}_o(X)
&= \varphi_{o}(\gamma)(\mathrm{Exp}_o(X))
= \gamma(\exp(X))H \\
&= \exp( (d\gamma)_{e}(X))H
= \mathrm{Exp}_o( (d\gamma)_{e}(X)).
\end{align*}
Thus we have $X=(d\gamma)_{e} (X)$,
since $(d\gamma)_{e}(X) \in U$ and $\mathrm{Exp}_o|_{U}$ is a diffeomorphism.
In particular, we have $X=(d\theta_{1})_{e}(X)=(d\theta_{2})_{e}(X)$.
Hence $X\in \mathfrak{k}_{1}\cap\mathfrak{k}_{2}$ holds.
Since $X\in U\subset \mathfrak{m}_{1}+\mathfrak{m}_{2}$, 
we can see that $X=0\in (\mathfrak{k}_{1}\cap\mathfrak{k}_{2}) \cap (\mathfrak{m}_{1}+\mathfrak{m}_{2})=\{0\}$ and $\mathrm{Exp}_o(X)=o$.
Therefore, $o=eH$ is an isolated fixed point of the action of $\varphi_{o}(\Gamma)$. 
By the $G$-equivariance law (\ref{eq:G-equivariance_law}),
we see that $gH$ is an isolated fixed point of the action of $\varphi_{gH}(\Gamma)$
for any point $gH \in G/H$. 
Therefore, $(G/H, \Gamma, \{\varphi_{gH}\}_{gH \in G/H})$ is a generalized $s$-manifold.


\section{Polars and antipodal sets of generalized $s$-manifolds}
\label{sec:polars_and_antipodal_sets_of_generalized_s-manifolds}

Chen and Nagano introduced the notions of polar and antipodal set for compact Riemannian symmetric spaces, 
and defined an invariant called 2-number as the supremum of the cardinality of antipodal sets (cf.\ \cite{Chen-Nagano1988}).
We define these notions for generalized $s$-manifolds.
In this section, 
let $(M, \Gamma, \{\varphi_x \}_{x \in M})$ be  a generalized $s$-manifold.

\begin{dfn}
For $x\in M$ and $ \gamma \in \Gamma$, 
the fixed point set $F(\varphi_x(\gamma), M)$ of the symmetric transformation $\varphi_x(\gamma)$ at $x$ for $\gamma$
is decomposed into the disjoint union of connected components $M^{\gamma}_{\lambda }$:
$$
F(\varphi_{x}(\gamma), M) = \bigcup_{\lambda \in \Lambda} M^{\gamma}_{\lambda }.
$$
Then each connected component $M^{\gamma}_{\lambda}$ is called a \textit{$\gamma$-polar} at $x$ in $M$. 
The connected component containing $x$ is especially denoted by $M_0^\gamma$.
\end{dfn}

\begin{dfn}
For $x \in M$,  the fixed point set $F(\varphi_x(\Gamma), M)$ of the symmetric transformation group $\varphi_x(\Gamma)$ at $x$
is decomposed into the disjoint union of connected components $M_{\lambda}$:
$$
F(\varphi_{x}(\Gamma), M) = \bigcup_{\lambda \in \Lambda} M_{\lambda}.
$$
Then each connected component $M_{\lambda}$ is called a \textit{polar} at $x$ in $M$.
In particular, a polar consisting of a single point is called a \textit{pole} at $x$. 
By the condition~(2) of Definition~\ref{dfn:generalized s-manifold}, 
the singleton $\{x\}$ becomes a pole at $x$, and it is termed as the \textit{trivial pole}. 
Let $M_0$ denote the trivial pole $\{x\}$.
\end{dfn}

\begin{dfn}
Let $A$ be a subset of $M$. 
\begin{enumerate}
\item For any $x, y \in A$ and $\gamma \in \Gamma$, 
if $\varphi_{x}(\gamma)(y)=y $
holds, that is, $y \in F(\varphi_x(\Gamma), M)$, 
we call $A$ an \textit{antipodal set} of $(M, \Gamma, \{\varphi_{x}\}_{x\in M})$.
\item If $A$ is an antipodal set and if any antipodal set containing $A$ coincides with $A$, then $A$ is called a \textit{maximal antipodal set}.  
\item Define
$$
\#_{\Gamma} (M) := \sup \{\# (A) \mid A\text{ is an antipodal set of $M$} \},
$$ 
then we call $\#_{\Gamma} (M)$ the \textit{antipodal number} of $(M, \Gamma, \{\varphi_{x}\}_{x\in M} )$. 
The antipodal number $\#_{\Gamma} (M)$ can be infinite.
\item A maximal antipodal set $A$ is said to be \textit{great} if  $\#_{\Gamma} (M)=\# (A)$.
\end{enumerate}
\end{dfn}

\begin{rem}
A great antipodal set is a maximal antipodal set, 
but a maximal antipodal set is not necessarily a great antipodal set. 
\end{rem}

Now we describe polars and antipodal sets
in some examples of generalized $s$-manifolds
given in Section~\ref{sec:constructions_and_examples}.

\subsubsection*{Example $1$} 

On the $2$-dimensional torus $T^2$,
we provided eight generalized $s$-structures (1)--(8).
For each case,
the fixed point set $F(\varphi_o(\Gamma),T^2)$ by the symmetric transformation group $\varphi_o(\Gamma)$ at $o$
is discrete,
and it is a great antipodal set as follows.

\begin{itemize}
\item In the cases (1) and (2)
\[
F(\varphi_o(\Gamma),T^2) = \left\{
[(0,0)],\
[(0,1/2)],\
[(1/2,0)],\
[(1/2,1/2)]
\right\}
\]

\item In the cases (3), (4), and (6)
\[
F(\varphi_o(\Gamma),T^2) = \left\{
[(0,0)],\
[(1/2,1/2)]
\right\}
\]

\item In the case (5)
\[
F(\varphi_o(\Gamma),T^2) = \left\{
[(0,0)],\
[(1/3,1/3)],\
[(2/3,2/3)]
\right\}
\]

\item In the cases (7), (8)
\[
F(\varphi_o(\Gamma),T^2) = \left\{
[(0,0)]
\right\}
\]
\end{itemize}
These examples show that 
antipodal sets and the antipodal number depend on generalized $s$-structures
on a manifold.

\subsubsection*{Examples $2$ and $3$}

We describe polars and antipodal sets of generalized flag manifolds.
Let $G$ be a compact connected semisimple Lie group.
For $X \in \mathfrak{g} \setminus \{0\}$,
we consider the generalized flag manifold $M:= \mathrm{Ad}(G)X$, that is the adjoint orbit of $G$ in $\mathfrak{g}$.
Then $M$ is diffeomorphic to the coset manifold $G/G_X$,
and we define a generalized $s$-structure $(\Gamma, \{ \varphi_x \circ f\}_{x \in G/G_X})$ on $G/G_X$,
where $\Gamma=\mathrm{Z}(G_x)_0$ as in Example~3 in Section~\ref{sec:constructions_and_examples}.
In particular, when $M=\mathrm{Ad}(G)X$ is a regular orbit, the isotropy subgroup $G_X$ is a maximal torus $T$ of $G$,
hence $M \cong G/T$ is a flag manifold as in Example~2.

We give the inner product on $\mathfrak{g}$ by the negative of the Killing form.
The orbit $M$ is a submanifold of the Euclidean space $\mathfrak{g}$.
For $Y \in \mathfrak{g}$, we define
\[
Y^*_X := \frac{d}{dt}\Big|_{t=0} \mathrm{Ad}(\exp tY)X = [Y,X] = -\mathrm{ad}(X)Y.
\]
Since $\mathfrak{g} \to T_XM;\ Y \mapsto Y^*_X = -\mathrm{ad}(X)Y$ is surjective,
$T_XM$ is identified with the linear subspace $\mathrm{Im}(\mathrm{ad}(X))$ in $\mathfrak{g}$.
And, by the orthogonal direct sum decomposition $\mathfrak{g} = \mathrm{Im}(\mathrm{ad}(X)) \oplus \mathrm{Ker}(\mathrm{ad}(X))$,
the normal space $T_X^\perp M$ is identified with $\mathrm{Ker}(\mathrm{ad}(X))$.

For $a \in \Gamma = \mathrm{Z}(G_X)_0$, the symmetric transformation $\psi_o(a):=(\varphi_o \circ f)(a) \in \mathrm{Diff}(G/G_X)$ at the origin $o \in G/G_X$
is defined by
\[
\psi_o(a) : G/G_X \to G/G_X;\ gG_X \mapsto \mathrm{I}_a(g)G_X = a(gG_X).
\]
Hence, via the diffeomorphism
\[
M=\mathrm{Ad}(G)X \to G/G_X;\ \mathrm{Ad}(g)X \mapsto gG_X,
\]
the symmetric transformation $\psi_X(a) \in \mathrm{Diff}(M)$ at $X \in M$ is given by
\[
\psi_X(a) : M \to M;\ Y \mapsto \mathrm{Ad}(a)Y.
\]

\begin{prop}[\cite{IIOST}, Theorem~2.2] \label{prop:antipodal_condition_in_generalized_flag_manifold}
A point $Y \in M$ is fixed by the symmetric transformation group $\psi_X(\Gamma)$ at $X \in M$
if and only if $[X,Y]=0$, that is,
\[
F(\psi_X(\Gamma), M)
= \{ Y \in M \mid [X,Y] =0 \}
= M \cap \mathrm{Ker}(\mathrm{ad}(X)).
\]
\end{prop}

In particular, when $M=\mathrm{Ad}(G)X$ is a regular orbit, 
that is, when $M\cong G/T$ is a flag manifold,
the Lie algebra $\mathfrak{t} = \mathrm{Ker}(\mathrm{ad}(X))$ of $T$
is a maximal abelian subalgebra of $\mathfrak{g}$.
Thus, from \cite[Proposition 2.2 p.~285]{Helgason},
we have
\[
F(\psi_X(\Gamma),M) = M \cap \mathfrak{t} = W(\Delta)X,
\]
where $W(\Delta)$ is the Weyl group of the root system $\Delta$ of $\mathfrak{g}$
with respect to $\mathfrak{t}$.
Hence $F(\psi_X(\Gamma),M)$ is a finite set, namely, all polars are poles.

In general,
for a generalized flag manifold $M=\mathrm{Ad}(G)X$, we have the following proposition.

\begin{prop} \label{prop:polar_of_generalized_flag_manifold}
Let $\mathfrak{t}$ be a maximal abelian subalgebra of $\mathfrak{g}$ containing $X$.
Then
\[
F(\psi_X(\Gamma),M) = \mathrm{Ad}(G_X)\big(W(\Delta)X\big),
\]
where $W(\Delta)$ is the Weyl group of the root system $\Delta$ of $\mathfrak{g}$
with respect to $\mathfrak{t}$.
For any polar $M_\lambda$ of $M$ at $X$,
there exists $w \in W(\Delta)$ such that $M_\lambda = \mathrm{Ad}(G_X)(wX)$.
\end{prop}

\begin{proof}
From Proposition~\ref{prop:antipodal_condition_in_generalized_flag_manifold},
we have $W(\Delta)X = M \cap \mathfrak{t} \subset F(\psi_X (\Gamma),M)$.
For each $w \in W(\Delta)$, we can verify
$[X, \mathrm{Ad}(g)(wX)] = 0$
for all $g \in G_X$. Hence
$F(\psi_X(\Gamma),M) \supset \mathrm{Ad}(G_X)\big( W(\Delta)X \big)$.

Conversely,
for any $Y' \in F(\psi_X(\Gamma),M)$,
there exists a maximal abelian subalgebra $\mathfrak{t}'$ of $\mathfrak{g}$ containing $X$ and $Y'$,
since $[X,Y']=0$ by Proposition~\ref{prop:antipodal_condition_in_generalized_flag_manifold}.
Since $\mathfrak{t}$ and $\mathfrak{t}'$ are also maximal abelian subalgebras of
$\mathfrak{g}_X$,
there exists $g \in G_X \subset G$ such that
$\mathfrak{t}' = \mathrm{Ad}_{G_X}(g)\mathfrak{t} = \mathrm{Ad}_{G}(g)\mathfrak{t}$
by the uniqueness of the maximal abelian subalgebra of $\mathfrak{g}_X$
under the action of $\mathrm{Ad}_{G_X}(G_X)$.
Thus there exists $Y \in \mathfrak{t}$ such that $Y' = \mathrm{Ad}_G(g)Y$.
Then, from \cite[Proposition 2.2 p.~285]{Helgason}, we have
$Y \in M \cap \mathfrak{t} = W(\Delta)X$,
hence $Y' \in \mathrm{Ad}_G(G_X)\big( W(\Delta)X \big)$.
Therefore $F(\psi_X(\Gamma),M) \subset \mathrm{Ad}(G_X)\big( W(\Delta)X \big)$.
Consequently, we obtain
$F(\psi_X(\Gamma),M) = \mathrm{Ad}(G_X)\big( W(\Delta)X \big)$.

Since $G_X$ is connected (cf.\ \cite{Borel}) and $W(\Delta)X$ is a finite set,
we deduce that $\mathrm{Ad}(G_X)(wX)$ is a connected component of $F(\psi_X(\Gamma),M)$
for each $w \in W(\Delta)$.
\end{proof}

\begin{rem}
In Proposition~\ref{prop:polar_of_generalized_flag_manifold},
since $wX \in \mathfrak{t} \subset \mathfrak{g}_X$,
we have $M_\lambda = \mathrm{Ad}_G(G_X)(wX) = \mathrm{Ad}_{G_X}(G_X)(wX)$,
that is an adjoint orbit of the compact Lie group $G_X$.
Thus the polar $M_\lambda$ is a generalized flag manifold itself,
if it is not a pole.
\end{rem}

Proposition~\ref{prop:antipodal_condition_in_generalized_flag_manifold}
indicates that the antipodal condition for points in a generalized flag manifold $M$
is the commutativity as elements of the Lie algebra $\mathfrak{g}$.
Thus, elements of an antipodal set $A$ of $M$ span an abelian subalgebra $\mathrm{span}_{\mathbb{R}}(A)$ of $\mathfrak{g}$.
Therefore we conclude the following proposition.

\begin{prop}[\cite{IIOST}, Theorem~2.2]
Let $A$ be a maximal antipodal set of the generalized flag manifold $M=\mathrm{Ad}(G)X$.
Then there exists a maximal abelian subalgebra $\mathfrak{t}$ of $\mathfrak{g}$
such that $A = M \cap \mathfrak{t}$.
Hence $A$ is an orbit $W(\Delta)Y$
of the Weyl group $W(\Delta)$ of the root system $\Delta$ of $\mathfrak{g}$
with respect to $\mathfrak{t}$, where $Y \in A$.
In particular, all maximal antipodal sets of $M$ are congruent
to each other under the action of $G$.
\end{prop}

\medskip
For a generalized $s$-manifold $(M, \Gamma, \{\varphi_x \}_{x \in M})$,
we define the notion of subspace similarly to the case of symmetric spaces. 

\begin{dfn}\label{dfn:subspace}
A submanifold $X$ of $M$ is said to be a \textit{subspace} of $M$
 if $\varphi_{x}(\gamma) (X)=X$ for any $x\in X, \gamma \in \Gamma$. 
\end{dfn}

\begin{prop} \label{prop:subspace}
Let $X$ be a subspace of $M$.
For each point $x \in X$ and element $\gamma \in \Gamma$, 
define $\widetilde\varphi_x(\gamma) := \varphi_x(\gamma)|_X$. 
Then $(X, \Gamma, \{\widetilde\varphi_{x}\}_{x \in X})$ is a generalized $s$-manifold. 
\end{prop}

\begin{proof}
Since $(M, \Gamma, \{\varphi_x \}_{x \in M})$ is a generalized $s$-manifold, 
$\widetilde\varphi_{x}: \Gamma \to \mathrm{Diff}(X)$ is a homomorphism for $x\in X$.
Then $(X, \Gamma, \{\widetilde\varphi_{x}\}_{x \in X})$ satisfies 
(1), (2) and (3) in Definition~\ref{dfn:generalized s-manifold}.
\end{proof}

\begin{rem}
Even if $(M, \Gamma, \{\varphi_x \}_{x \in M})$ is effective, 
its subspace $(X, \Gamma, \{\widetilde\varphi_{x}\}_{x \in X})$ is not necessarily effective.
\end{rem}

\begin{cor}
If $X$ is a subspace of $(M, \Gamma, \{\varphi_x\}_{x \in M})$, 
then from Proposition~\ref{prop:subspace}, 
$(X, \Gamma, \{\widetilde\varphi_x\}_{x\in X})$ is a generalized $s$-manifold. 
Then an antipodal set of $X$ is that of $M$, 
and $$\#_{\Gamma} (X) \leq \#_{\Gamma} (M)$$ holds.
\end{cor}
\begin{proof}
If $A$ is an antipodal set of a subspace $X$, 
then for any $x, y\in A, \gamma \in \Gamma$, 
$$
y=\widetilde{\varphi}_{x}(\gamma)(y)=\varphi_{x}(\gamma)(y),
$$
and $A$ is also an antipodal set of $M$. 
Therefore, we obtain 
$$
\#_{\Gamma} (X) \leq \#_{\Gamma} (M).
$$
\end{proof}

In a generalized $s$-manifold, 
we explain a sufficient condition for polars and $\gamma$-polars to be subspaces.
In general, a connected component of the fixed point set of a diffeomorphism is not necessarily a submanifold.
When a generalized $s$-manifold $M$ has a Riemannian metric and 
symmetric transformations are isometries,
we can see  that polars and $\gamma$-polars are submanifolds of $M$ from the following lemma.

\begin{lem}[\cite{Kobayashi}, Ch.~2 Theorem~5.1  in p.~69]\label{Thm:kobayashi}
Let $M$ be a Riemannian manifold and $G$ any set of isometries of $M$.  
Let $F$ be the set of points of $M$ which are left fixed by all elements of $G$.  
Then each connected component of $F$ is a closed totally geodesic submanifold of $M$.
\end{lem}

If the action of symmetric transformation group $\varphi_{x}(\Gamma)$ at a point $x$ in $M$ is proper, 
then $M$ admits a Riemannian metric which is invariant under the action of  $\varphi_{x}(\Gamma)$ by the following lemma.

\begin{lem}[\cite{Koszul}, Ch.~1 Theorem~2  in p.~9]\label{Thm:koszul}
Let $G$ be a Lie group acting properly and differentiably on a paracompact differentiable manifold $X$. 
Then $X$ admits a Riemannian metric invariant under $G$.
\end{lem}

Therefore, we have the following proposition.

\begin{prop}
Let $\Gamma $ be a Lie group acting properly and differentiably on a paracompact differentiable manifold $M$.
Then each connected component of the fixed point set of the action of $\Gamma$ is  a submanifold of $M$.
\end{prop}

Using the above proposition,
we can show the following lemma for a generalized $s$-manifold $(M, \Gamma, \{\varphi_x \}_{x \in M} )$. 
Hereafter we suppose that $M$ is paracompact.

\begin{lem} \label{lem:polar-submfd}
Let $\Gamma$ be a Lie group or discrete group.
Suppose that the action of the symmetric transformation group $\varphi_{x} (\Gamma)$
at a point $x$ in $M$ is proper.
\begin{enumerate}
\item Each $\gamma$-polar $M_{\lambda}^{\gamma}$ at $x$ is a submanifold of $M$ for any $\gamma \in \Gamma$.
\item Each polar $M_{\lambda}$ at $x$ is a submanifold of $M$.
\end{enumerate}

\end{lem}

\begin{prop}\label{prop:polar-subsp}
Let $\Gamma$ be an abelian Lie group or abelian discrete group.
Suppose that the action of the symmetric transformation group $\varphi_{x} (\Gamma)$
at a point $x$ in $M$ is proper.
\begin{enumerate}
\item Each $\gamma$-polar $M^{\gamma}_{\lambda}$ at $x$ is a subspace of $M$ for any $\gamma \in \Gamma$.
\item Each polar $M_{\lambda}$ at $x$ is a subspace of $M$.
\end{enumerate}

\end{prop}

\begin{proof}
(1) By Lemma~\ref{lem:polar-submfd} (1), $M^{\gamma}_{\lambda}$ is a submanifold.
We show that  
$\varphi_{y}(\delta) (M^{\gamma}_{\lambda}) = M^{\gamma}_{\lambda}$
holds for every $y \in M^{\gamma}_{\lambda}$ and $\delta \in \Gamma$.
Since
$$
\varphi_{x}(\gamma) \circ \varphi_{y}(\delta) \circ \varphi_{x}(\gamma)^{-1}
=\varphi_{\varphi_{x}(\gamma)(y)}(\gamma \delta \gamma^{-1})
=\varphi_{y}(\delta),
$$
we have 
$$
\varphi_{x}(\gamma) \circ \varphi_{y}(\delta)(z)
=\varphi_{y}(\delta) \circ \varphi_{x}(\gamma)(z)
=\varphi_{y}(\delta)(z)
$$
for each $z \in  M^{\gamma}_{\lambda}$. 
Hence, 
$\varphi_{y}(\delta)(z) \in F(\varphi_{x}(\gamma), M)$. 
Thus, we have
$\varphi_{y}(\delta)(M^{\gamma}_{\lambda}) \subset F(\varphi_{x}(\gamma), M)$
and  
$\varphi_{y}(\delta)(M^{\gamma}_{\lambda}) \subset M^{\gamma}_{\lambda'}$
for some $\lambda'$. 
Then, since $y \in M^{\gamma}_{\lambda} \cap \varphi_{y}(\delta)(M^{\gamma}_{\lambda})$,
we can see that
$\varphi_{y}(\delta)(M^{\gamma}_{\lambda}) \subset M^{\gamma}_{\lambda}$.
In the same way, it can be shown that
$\varphi_{y}(\delta^{-1})(M^{\gamma}_{\lambda}) \subset M^{\gamma}_{\lambda}$.
Therefore, we obtain
$\varphi_{y}(\delta)(M^{\gamma}_{\lambda}) = M^{\gamma}_{\lambda}$.

The proof of (2) is almost the same as that of (1).
By Lemma~\ref{lem:polar-submfd} (2), $M_{\lambda}$ is a submanifold.
We show $\varphi_{y}(\delta) (M_{\lambda}) = M_{\lambda}$
for each $y \in M_{\lambda}$ and $\delta \in \Gamma$. 
Since 
$$
\varphi_{x}(\gamma) \circ \varphi_{y}(\delta) \circ \varphi_{x}(\gamma)^{-1}
=\varphi_{\varphi_{x}(\gamma)(y)}(\gamma \delta \gamma^{-1})
=\varphi_{y}(\delta)
\qquad (\forall \gamma \in \Gamma),
$$
we have 
$$
\varphi_{x}(\gamma) \circ \varphi_{y}(\delta)(z)
=\varphi_{y}(\delta) \circ \varphi_{x}(\gamma)(z)
=\varphi_{y}(\delta)(z)
\qquad (\forall \gamma \in \Gamma)
$$
for each $z \in M_{\lambda}$.
Thus, 
$\varphi_{y}(\delta)(z) \in F(\varphi_{x}(\Gamma), M)$. 
Then we have 
$\varphi_{y}(\delta)(M_{\lambda}) \subset F(\varphi_{x}(\Gamma), M)$
and 
$\varphi_{y}(\delta)(M_{\lambda}) \subset M_{\lambda'}$
for some $\lambda'$. 
Since $y \in M_{\lambda} \cap \varphi_{y}(\delta)(M_{\lambda})$, 
we can see $\varphi_{y}(\delta)(M_{\lambda}) \subset M_{\lambda}$.
In the same way, we have $\varphi_{y}(\delta^{-1})(M_{\lambda}) \subset M_{\lambda}$.
Therefore, we obtain $\varphi_{y}(\delta)(M_{\lambda}) = M_{\lambda}$.
\end{proof}

The following lemma is clear from the definition of antipodal set.
\begin{lem} \label{lem:antipodal subset fixed point set}
Let $A$ be an antipodal set of $M$ and $x \in A$.
Then the following inclusion relations hold.
\begin{enumerate}
\item $A \subset F \big( \varphi_x(\gamma), M \big) \quad (\forall \gamma \in \Gamma)$,
\item $A \subset F \big( \varphi_x(\Gamma), M \big)$.
\end{enumerate}
\end{lem}

Finally, we have the following inequalities for the antipodal number of a generalized $s$-manifold.

\begin{thm}\label{thm:inequality of antipodal number}
Let $\Gamma$ be an abelian Lie group or abelian discrete group.
Let $A$ be a great antipodal set of $(M, \Gamma, \{\varphi_x\}_{x\in M})$.
Suppose that the action of the symmetric transformation group $\varphi_{x} (\Gamma)$
at a point $x$ in $A$ is proper.
\begin{enumerate}
\item 
For an element $\gamma \in \Gamma$,
if $F(\varphi_{x}(\gamma), M)$ is the disjoint union $F(\varphi_{x}(\gamma), M) = \bigcup_{i=0}^{r}M_{i}^{\gamma}$ 
of a finite number of $\gamma$-polars at $x$,
then the following holds:
$$ 
\#_{\Gamma}(M) \leq \sum_{i=0}^{r} \#_{\Gamma}(M_{i}^{\gamma}). 
$$ 
\item
If $F(\varphi_{x}(\Gamma), M)$ is the disjoint union 
$F(\varphi_{x}(\Gamma), M) = \bigcup_{i=0}^{r}M_{i}$
of a finite number of polars at $x$,
then the following holds:
$$
\#_{\Gamma}(M) \leq \sum_{i=0}^{r} \#_{\Gamma}(M_{i}).
$$
\end{enumerate}

\end{thm}

\begin{proof}
(1) 
Proposition~\ref{prop:polar-subsp} (1) shows that each $\gamma$-polar $M_{i}^{\gamma}$ is a subspace of $(M, \Gamma, \{\varphi_x\}_{x\in M})$.
Thus $(M_{i}^{\gamma}, \Gamma, \{\widetilde\varphi_x\}_{x \in M_{i}^{\gamma}})$ is a generalized $s$-manifold.
From Lemma~\ref{lem:antipodal subset fixed point set} (1), 
the great antipodal set $A$ of $M$ is a subset of $F(\varphi_{x}(\gamma), M)$.
Then, $A$ is decomposed into the disjoint union of $A_{i}^{\gamma}:= A\cap M_{i}^{\gamma}$,
i.e.
$$
A = \bigcup_{i=0}^{r} A_{i}^{\gamma}.
$$
Since $A_{i}^{\gamma}$ is an antipodal set of $M_{i}^{\gamma}$, we obtain 
$$
\#_{\Gamma}(M)
= \#(A)
= \sum_{i=0}^{r} \#(A_{i}^{\gamma})
\leq \sum_{i=0}^{r} \#_{\Gamma}(M_{i}^{\gamma}).
$$ 

The proof of (2) is similar to (1).
Proposition~\ref{prop:polar-subsp} (2) shows that each polar $M_{i}$ is a subspace of $(M, \Gamma, \{\varphi_x\}_{x\in M})$.
Thus $(M_{i}, \Gamma, \{\widetilde\varphi_x\}_{x \in M^{\gamma}})$ is a generalized $s$-manifold. 
From Lemma~\ref{lem:antipodal subset fixed point set} (2), 
the great antipodal set $A$ of $M$ is a subset of $F(\varphi_{x}(\Gamma), M)$.
Then, $A$ is decomposed into the disjoint union of $A_{i}:= A\cap M_{i}$,
i.e.
$$
A = \bigcup_{i=0}^{r} A_{i}.
$$
Since $A_{i}$ is an antipodal set of $M_{i}$, we obtain 
$$
\#_{\Gamma}(M)
= \#(A)
= \sum_{i=0}^{r} \#(A_{i})
\leq \sum_{i=0}^{r} \#_{\Gamma}(M_{i}).
$$ 
\end{proof}


\end{document}